\makeatletter\def\input@path{{corps/}{styles/}}\makeatother
\documentclass{madani}

\usepackage{color}
\usepackage{graphics}
\usepackage[all,2cell,ps]{xy}
\UseAllTwocells
\UsePSspecials{}

\numberwithin{equation}{section}

\begin{document}
\pagenumbering{arabic} 
\setcounter{page}{1}
\title{Generalization of a going-down theorem in the category of Chow-Grothendieck motives due to N. Karpenko}
\author{Charles De Clercq}

\maketitle

\begin{abstract}Let $\mathbb{M}:=(M(X),p)$ be a direct summand of the motive associated with a geometrically split, geometrically variety over a field $F$ satisfying the nilpotence principle. We show that under some conditions on an extension $E/F$, if $\mathbb{M}$ is a direct summand of another motive $M$ over an extension $E$, then $\mathbb{M}$ is a direct summand of $M$ over $F$.
\end{abstract}

\section{Introduction}
Let $\Lambda$ be a finite commutative ring. Our main reference on the category $CM(F;\Lambda)$ of Chow-Grothendieck motives with coefficients in $\Lambda$ is \cite{EKM}.

The purpose of this note is to generalize the folowing theorem due to N. Karpenko (\cite{hhfba}, proposition 4.5). Throughout this paper we understand a $F$-variety over a field $F$ as a separated scheme of finite type over $F$.

\begin{theoreme}\label{4.5}Let $\Lambda$ be a finite commutative ring. Let $X$ be a geometrically split, geometrically irreducible $F$-variety satisfying the nilpotence principle. Let $M\in CM(F;\Lambda)$ be another motive. Suppose that an extension $E/F$ satisfies
\begin{enumerate}
\item the $E$-motive $M(X)_E\in CM(E;\Lambda)$ of the $E$-variety $X_E$ is indecomposable;
\item the extension $E(X)/F(X)$ is purely transcendental;
\item the motive $M(X)_E$ is a direct summand of the motive $M$.
\end{enumerate}
Then the motive $M(X)$ is a direct summand of the motive $M$.
\end{theoreme}

We generalize this theorem when the motive $M(X)\in CM(F;\Lambda)$ is replaced by a direct summand $(M(X),p)$ associated with a projector $p\in End_{CM(F;\Lambda)}(M(X))$. The proof given by N. Karpenko in \cite{hhfba} cannot be used in the case where $M(X)$ is replaced by a direct summand because of the use on the \emph{multiplicity} (\cite{EKM}, §75) as the multiplicity of a projector in the category $CM(F;\Lambda)$ is not always equal to $1$ (and it can even be $0$). The proof given here for its generalization gives also another proof of theorem \ref{4.5}.

\section{Suitable basis of the dual module of a geometrically split $F$-variety}

Let $X$ be a geometrically split, geometrically irreductible $F$-variety satisfying the nilpotence principle. We note $CH(\overline{X};\Lambda)$ as the colimit of the $CH(X_{K};\Lambda)$ over all extensions $K$ of $F$. By assumption there is a integer $n=rk(X)$ such that $$CH(\overline{X};\Lambda)\simeq\bigoplus_{i=0}^n\Lambda.$$ Let $(x_i)_{i=0}^n$ be a base of the $\Lambda$-module $CH(\overline{X};\Lambda)$. Each element $x_i$ of the basis is associated with a subvariety of $X_E$, where $E$ is a splitting field of $X$. We note $\varphi(i)$ for the dimension of the $E$-variety associated to $x_i$.

\begin{proposition}Let $X$ be a geometrically split $F$-variety. Then the pairing
$$\fonction{\Psi}{CH(\overline{X};\Lambda)\times CH(\overline{X};\Lambda)}{\Lambda}{(\alpha,\beta)}{deg(\alpha\cdot\beta)}$$
is bilinear, symetric and non-degenerate.
\end{proposition}

The pairing $\Psi$ induces an isomorphism between $CH(\overline{X};\Lambda)$ and its dual module $Hom_{\Lambda}(CH(\overline{X};\Lambda),\Lambda)$. This isomorphism is given by
$$\fonction{}{CH(\overline{X};\Lambda)}{Hom_{\Lambda}(CH(\overline{X};\Lambda),\Lambda}{x}{\Psi(x,\cdot)}$$

Considering the inverse images of the dual basis of $Hom_{\Lambda}(CH(\overline{X},\Lambda);\Lambda)$ associated with the basis $x_i$, we get another basis $(x_i^{\ast})_{i=0}^n$ of $CH(\overline{X};\Lambda)$ such that $$\Psi(x_i,x_j^{\ast})=\delta_{ij}$$ where $\delta_{ij}$ is the usual Kronecker symbol.

\begin{proposition}\label{64.3}Let $M$ and $N$ be two motives in $CM(F;\Lambda)$ such that $M$ is split. Then there is an isomorphism
$$CH^{\ast}(M;\Lambda)\otimes CH^{\ast}(N;\Lambda)\longrightarrow CH^{\ast}(M\otimes N;\Lambda)$$
\end{proposition}

\begin{proof}c.f. \cite{EKM} proposition 64.3.
\end{proof}

Let $Y$ be a smooth complete irreducible $F$-variety. We note $M$ for the motive $(M(Y),q)$ associated with a projector $q\in End(M(Y))$. Then we have the following computations.

\begin{lemme}\label{calculs}For any integers $i$, $j$, $k$ and $s$ less than $rk(X)=n$, and for any cycles $y$ and $y'$ in $CH(\overline{Y};\Lambda)$, with
$1$ being the identity class in either $CH(\overline{X};\Lambda)$ or $CH(\overline{Y};\Lambda)$ we have
\begin{enumerate}
\item $(x_i\times x_j^{\ast})\circ (x_k\times x_s^{\ast})=\delta_{is} (x_k\times x_j^{\ast})$
\item $(x_i\times y\times 1)\circ (x_k\times x_s^{\ast})=\delta_{is} (x_k\times y\times 1)$
\item $(y'\times x_j^{\ast})\circ (x_i\times y)=\deg(y'\cdot y) (x_i\times x_j^{\ast})$
\end{enumerate}
\end{lemme}
\begin{proof}We only compute (2) (other cases are similar).
\begin{align}
(x_i\times y\times 1)\circ (x_k\times x_s^{\ast})&=(^{\overline{X}}\!\!p^{\overline{Y}\times \overline{X}}_{\overline{X}})_{\ast}((^{\overline{X}\times\overline{X}}\!\!p_{\overline{Y}\times\overline{X}})^{\ast}(x_k\times x_s^{\ast})\cdot (p^{\overline{X}\times\overline{Y}\times \overline{X}}_{\overline{X}})^{\ast}(x_i\times y\times 1))\\
&= (^{\overline{X}}\!\!p^{\overline{Y}\times \overline{X}}_{\overline{X}})_{\ast}((x_k\times x_s^{\ast}\times 1\times 1)\cdot(1\times x_i\times y\times 1))\\
&= (^{\overline{X}}\!\!p^{\overline{Y}\times \overline{X}}_{\overline{X}})_{\ast}(x_k\times(x_s^{\ast}\cdot x_i)\times y\times 1)\\
&= \delta_{is}(x_k\times y\times 1)
\end{align}
\end{proof}

\section{Rational cycles of a geometrically split $F$-variety}

Let $X$ be a geometrically split $F$-variety. We note $(M(X),p)$ the direct summand of $M(X)$ associated with a projector $p\in CH_{\dim(X)}(X\times X;\Lambda)$. Considering the motive $M$ defined in the previous section, if $(M(X_E),p_E)$ is a direct summand of $M_E$ for some extension $E/F$, then there exists cycles $f\in CH(X_E\times Y_E;\Lambda)$ and $g\in CH(Y_E\times X_E;\Lambda)$ such that $f\circ g=p_E$. We can write these cycles in suitable basis of $CH(\overline{X}\times \overline{Y};\Lambda)$, $CH(\overline{Y}\times \overline{X};\Lambda)$ and $CH(\overline{X}\times \overline{X};\Lambda)$ by proposition \ref{64.3}. Thus there are two subsets $F$ and $G$ of $\{0,\dots,n\}$, scalars $f_i$, $g_j$, $p_{ij}$ and cycles $y_i$, $y'_j$ in $CH(\overline{Y};\Lambda)$ such that

\begin{enumerate}
\item $\overline{f}=\sum_{i\in F}f_i(x_i\times y_i)$
\item $\overline{g}=\sum_{j\in G}g_j(y'_j\times x^{\ast}_j)$
\item $\overline{p}=\sum_{i\in F}\sum_{j\in G}p_{ij}(x_i\times x_j^{\ast})$
\end{enumerate}

With $p_{ij}=f_ig_j\deg(y_j'\cdot y_i)$ by lemma \ref{calculs} as $g\circ f =p_E$.

\begin{notation}Let $p\in CH_{\dim(X)}(X\times X)$ be a non-zero projector. Considering $\overline{p}$, the image of $p$ in a splitting field of the $F$-variety $X$, we can write $\overline{p}=\sum_{i\in P_1}\sum_{j\in P_2}p_{ij}(x_i\times x_j^{\ast})$. We define the \emph{least codimension} of $p$ (denoted $cdmin(p)$) by
$$cdmin(p):=\min_{(i,j),~p_{ij}\neq 0}(\dim(\overline{X})-\varphi(i))$$
\end{notation}

\begin{proposition}\label{proj}Let $p\in CH_{\dim(X)}(X\times X)$ be a non-zero projector. We consider its decomposition $\overline{p}=\sum_{i\in P_1}\sum_{j\in P_2}p_{ij}(x_i\times x_j^{\ast})$ in a splitting field of $X$. Then for any $i\in P_1$ and $j\in P_2$ we have
$$p_{ij}=\sum_{k\in P_1\cap P_2} p_{kj}p_{ik}$$
\end{proposition}
\begin{proof}We can assume that $\varphi(i)$ is constant on $P_1$. Then a straightforward computation gives
\begin{align}\overline{p}\circ \overline{p}&=(\sum_{i\in P_1}\sum_{j\in P_2}p_{ij}(x_i\times x_j^{\ast}))\circ (\sum_{k\in P_1}\sum_{s\in P_2}p_{ij}(x_i\times x_j^{\ast}))\\
&= \sum_{i\in P_1}\sum_{j\in P_2}\sum_{k\in P_1}\sum_{s\in P_2} p_{ij}p_{ks}(x_i\times x_j^{\ast})\circ(x_k\times x_s^{\ast})\\
&= \sum_{i\in P_1}\sum_{j\in P_2}\sum_{k\in P_1}\sum_{s\in P_2} p_{ij}p_{ks}\delta_{is}(x_k\times x_j^{\ast})\\
&= \sum_{k\in P_1}\sum_{s\in P_2}\left(\sum_{i\in P_1\cap P_2}p_{ij}p_{ki}(x_k\times x_s^{\ast})\right)
\end{align}
Moreover $p\circ p=p$, thus if $(k,s)\in P_1\times P_2$ we have $p_{ks}=\sum_{i\in P_1\cap P_2} p_{is}p_{ki}$.

\end{proof}

\section{General properties of Chow groups}

Embedding the Chow group of the $F$-variety $X$ is quite usefull for computations, but the generalization of the theorem \ref{4.5} needs a direct construction of some $F$-rational cycles $f$ and $g$. We study in this section some properties of rationnal elements in Chow groups and how they behave when the extension $E(X)/F(X)$ is purely transcendental.

\begin{proposition}\label{trans}Let $Y$ be an $F$-varieties. Let $E/F$ be a purely transcendental extension. Then the morphism $$res_{E/F}:CH(Y;\Lambda)\longrightarrow CH(Y_E;\Lambda)$$ is an epimorphism. 
\end{proposition}
\begin{proof}Indeed the morphism $res_{E/F}$ coincides with the composition
$$CH(Y;\Lambda)\longrightarrow CH(Y\times \mathbb{A}^n_F;\Lambda)\longrightarrow CH(Y_E;\Lambda).$$
As the extension $E/F$ is purely transcendental, there is an isomorphism between $E$ and the function field of an affine space $\mathbb{A}^n_F$ for some integer $n$. The first map is an epimorphism by the homotopy invariance of Chow groups (\cite{EKM}, theorem 57.13) and the second map is an epimorphism as well (\cite{EKM}, corollary 57.11).
\end{proof}

\section{Generalization of the going-down theorem in the category of Chow-Grothendieck motives}

We now have all the material needed to prove the generalization of theorem \ref{4.5}.
\begin{theoreme}
Let $\Lambda$ be a finite commutative ring. Let $X$ be a geometrically split, geometically irreducible $F$-variety satisfying the nilpotence principle. Let also $M\in CM(F;\Lambda)$ be a motive. Suppose that an extension $E/F$ satisfies
\begin{enumerate}
\item the $E$-motive $(M(X)_E,p_E)$ associated with the $E$-variety $X_E$ is indecomposable;
\item the extension $E(X)/F(X)$ is purely transcendental;
\item the motive $(M(X_E),p_E)$ is a direct summand of the $E$-motive $M_E$.
\end{enumerate}
Then the motive $(M(X),p)$ is a direct summand of the motive $M$.
\end{theoreme}
\begin{proof}
We can consider that $M=(Y,q)$ for some smooth complete $F$-variety $Y$ and a projector $q\in CH_{\dim(Y)}(Y\times Y;\Lambda)$. If $p$ is equal to $0$ then the motive $(M(X),p)$ is the $0$ motive and $(M(X),p)$ is a direct summand of $M$. Now suppose that $p$ is not equal to $0$.

As $(M(X)_E,p_E)$ is a direct summand of $M_E$, there are $E$-rationnal cycles  $f\in CH_{\dim(X_E)}(X_E\times~Y_E;\Lambda)$ and $g\in CH_{\dim(Y_E)}(Y_E\times~X_E;\Lambda)$ such that $g\circ f=p_E$. We can decompose the images of these cycles in a splitting field of $X$ in suitable basis for computations :
\begin{enumerate}
\item $\overline{f}=\sum_{i\in F}f_i(x_i\times y_i)$
\item $\overline{g}=\sum_{j\in G}g_j(y'_j\times x^{\ast}_j)$
\item $\overline{p}=\sum_{i\in F}\sum_{j\in G}p_{ij}(x_i\times x_j^{\ast})$
\end{enumerate}
with $p_{ij}=f_ig_j\deg(y'_j\cdot y_i)$. 

Splitting terms whose first codimension is minimal in $\overline{f}$ and $\overline{p}$ by introducing $$F_1:=\{i\in F,~\varphi(i)=cdmin(p)\}$$ we get
\begin{enumerate}
\item $\overline{f}=\sum_{i\in F_1}f_i(x_i\times y_i)+\sum_{i\in F\setminus F_1}f_i(x_i\times y_i)$
\item $\overline{p}=\sum_{i\in F_1}\sum_{j\in G}p_{ij}(x_i\times x_j^{\ast})+\sum_{i\in F\setminus F_1}\sum_{j\in G}p_{ij}(x_i\times x_j^{\ast})$
\end{enumerate}

As $E(X)$ is an extension of $E$, the cycle $\overline{f}$ is $E(X)$-rational. Proposition \ref{trans} implies that the change of field $res_{E(X)/F(X)}$ is an epimorphism, hence $\overline{f}$ is an $F(X)$-rational cycle.

Considering the morphism $Spec(F(X))\longrightarrow X$ associated with the generic point of the geometrically irreducible variety $X$, we get a morphism 
$$\epsilon:(X\times Y)_{F(X)}\longrightarrow X\times Y\times X$$
This morphism induces a pull-back $\epsilon^{\ast}:CH_{\dim(X)}(\overline{X}\times \overline{Y}\times\overline{X};\Lambda)\longrightarrow CH_{\dim(X)}(\overline{X}\times \overline{Y};\Lambda)$
mapping any cycle of the form $\alpha\times \beta\times 1$ on $\alpha\times \beta$ and vanishing on elements $\alpha\times \beta\times\gamma$ if codim($\gamma$)$>0$. Moreover $\epsilon^{\ast}$ induces an epimorphism of $F$-rational cycles onto $F(X)$-rational cycles (\cite{EKM}, corollary 57.11). We can thus choose a $F$-rational cycle $f_1\in CH_{\dim(X)}(\overline{X}\times\overline{Y}\times\overline{X};\Lambda)$ such that $\epsilon^{\ast}(f_1)=\overline{f}$
.

By the expression of the pull-back $\epsilon^{\ast}$ we can assume 
$$\overline{f_1}=\sum_{i\in F_1}f_i(x_i\times y_i\times 1)+\sum_{i\in F\setminus F_1}f_i(x_i\times y_i\times 1)+\sum(\alpha\times \beta\times \gamma)$$
where the codimension of the cycles $\gamma$ is non-zero.

Considering $f_1$ as a correspondance from $\overline{X}$ to $\overline{X}\times\overline{Y}$, we consider $f_2:=f_1\circ p$ which is also a $F$-rational cycle. We have
\begin{align}\overline{f_2}&= (\sum_{i\in F_1}f_i(x_i\times y_i\times 1))\circ(\sum_{i\in F_1}\sum_{j\in G}p_{ij}(x_i\times x_j^{\ast}))+\sum_{i\in F\setminus F_1}\sum_{j\in G} \lambda_{ij} (x_i\times y_j\times 1)+\sum \tilde{\alpha}\times\tilde{\beta}\times\tilde{\gamma}\\
&= \sum_{i\in F_1}\sum_{j\in F_1\cap G}f_jp_{ij}(x_i\times y_j\times 1)+\sum_{i\in F\setminus F_1}\sum_{j\in G} \lambda_{ij} (x_i\times y_j\times 1)+\sum \tilde{\alpha}\times\tilde{\beta}\times\tilde{\gamma}
\end{align}
where the cycles $\tilde{\gamma}$ are of non-zero codimension, the cycles $\tilde{\alpha}$ are such that $codim(\tilde{\alpha})\geq~cdmin(p)$ and where elements $\lambda_{ij}$ are scalars.

We now consider the diagonal embedding
$$\fonction{\Delta}{\overline{X}\times \overline{Y}}{\overline{X}\times \overline{Y}\times \overline{X}}{(x,y)}{(x,y,x)}$$

The morphism $\Delta$ induces a pull-back $\Delta^{\ast}:CH_{\dim(X)}(\overline{X}\times\overline{Y}\times\overline{X};\Lambda)\longrightarrow CH_{\dim(X)}(\overline{X}\times \overline{Y};\Lambda)$

We note $f_3:=\Delta^{\ast}(f_2)$, which is also a $F$-rational cycle and whose expression in a splitting field of $X$ is
$$f_3=\sum_{i\in F_1}\sum_{j\in F_1\cap G}f_jp_{ij}(x_i\times y_j)+\sum_{i\in F\setminus F_1}\sum_{j\in G} \lambda_{ij} (x_i\times y_j)+\sum (\tilde{\alpha}\cdot\tilde{\gamma})\times\tilde{\beta}$$
where $codim(\tilde{\alpha}\cdot\tilde{\gamma})>cdmin(p)$ as $codim(\tilde{\alpha})\geq cdmin(p)$ and $codim(\tilde{\gamma})>0$.

We can compute the $g\circ f_3$:

\begin{align}\overline{g}\circ \overline{f_3}&= \overline{g}\circ (\sum_{i\in F_1}\sum_{j\in G}f_jp_{ij}(x_i\times y_j))+ \overline{g}\circ (\sum_{i\in F\setminus F_1}\sum_{j\in G} \lambda_{ij} (x_i\times y_j))+\overline{g}\circ(\sum (\tilde{\alpha}\cdot\tilde{\gamma})\times\tilde{\beta}))\\
&= \sum_{i\in F_1}\sum_{s\in G}\sum_{j\in F_1\cap G}g_sf_jp_{ij}(y'_s\times x_s^{\ast})\circ(x_i\times y)+(\sum\overline{\alpha}\times \overline{\beta})
\end{align}
With cycles $\overline{\alpha}$ such that $codim(\overline{\alpha})>cdmin(p)$. 
Computing the component of $g\circ f_3$ for elements of the form $x_k\times x_s^{\ast}$ with $\varphi(k)=cdmin(p)$ we get

\begin{align}\overline{g}\circ \overline{f_3}&= \sum_{i\in F_1}\sum_{s\in G}\sum_{j\in F_1\cap G}g_sf_jp_{ij}(y'_s\times x_s^{\ast})\circ(x_i\times y_j)+(\sum\overline{\alpha}\times \overline{\beta})\\
&=  \sum_{i\in F_1}\sum_{s\in G}\sum_{j\in F_1\cap G}g_sf_jp_{ij}\deg(y'_s\cdot y_j)(x_i\times x_s^{\ast})
\end{align}

Now we can see that if $k\in F_1$, then the coefficient of $g\circ f_3$ relatively to an element $x_k\times x_s^{\ast}$ is equal to $\sum_{i\in F_1\cap G}g_{s}f_ip_{ki}\deg(y_i\cdot y'_s)$. Moreover proposition \ref{proj} says that
$$\sum_{i\in F_1\cap G}g_{s}f_ip_{ki}\deg(y_i\cdot y'_s)=\sum_{i\in F_1\cap G}p_{is}p_{ki}=p_{ks}$$

Since $p$ is non-zero, there exists $(k,s)$ with $k\in F_1$ and $p_{ks}\neq 0$, thus we have shown that the cycle $g\circ f_3$ as a decomposition
$$g\circ f_3= p_{ks}(x_k\times x_s^{\ast})+\sum_{(i,j)\neq (k,s)}p_{ij}(x_i\times x_j^{\ast})+\sum(\overline{\alpha}\circ\overline{\beta})$$
where $codim(\overline{\alpha})>cdmin(p)$. Since $p$ is a projector, for any integer $n$ the $n$-th power of $g\circ f_3$ as always a non-zero component relatively to $x_k\times x_s^{\ast}$ which is equal to $p_{ks}$, that is to say
$$\forall n\in \mathbb{N},~(g\circ f_3)^{\circ n}=p_{ks}(x_k\times x_s^{\ast})+\sum_{(i,j)\neq (k,s)}p_{ij}(x_i\times x_j^{\ast})+\sum(\overline{\alpha}\circ\overline{\beta})$$
where $codim(\overline{\alpha})>cdmin(p)$.

As the ring $\Lambda$ is finite, there is a power of $g\circ (f_3)_E$ which is a non-zero idempotent (cf \cite{hhfba} lemma 3.2). Since the $E$-motive $(M(X)_E,p_E)$ is indecomposable this power of $g\circ (f_3)_E$ is equal to $p_E$. Thus we have shown that there exists an integer $n_1$ such that
$$(g\circ (f_3)_E)^{\circ n_1}=p_E$$
In particular if $g_1:=(g\circ (f_3)_E)^{\circ n_1-1}\circ g$ we get $g_1\circ (f_3)_E=p_E$.

Now we can transpose the last equality and get
$$^t(f_3)_E\circ ^t\!\!g_1=^t\!\!p_E.$$
Repeating the same process as before, we get a $F$-rational cycle $\tilde{g}$ and an integer $n_2$ such that
$$(^t(f_3)_E\circ (\tilde{g})_E)^{\circ n_2}=^t\!\!p_E$$
Now setting $\hat{g}:=(^t\!\!\tilde{g}\circ (f_3))^{\circ n_2-1}\circ ^t\!\!\tilde{g}$, we have two $F$-rational cycles $\hat{g}$ and $f_3$ such that
$$\hat{g}_E\circ (f_3)_E=p_E$$

Using the nilpotence principle again, there is an integer $\overline{n}\in \mathbb{N}$ such that
$$(\hat{g}\circ f_3)^{\overline{n}}=p$$
Hence if $\hat{f}:=f_3\circ(\hat{g}\circ f_3)^{\overline{n}-1}$, $\hat{f}$ is a $F$-rational cycle satisfying
$$\hat{g}\circ\hat{f}=p$$
Thus we have shown that the motive $(M(X),p)$ is a direct summand of the motive $M$.
\end{proof}

\bibliographystyle{plain}
\bibliography{madanipubli}

\begin{thebibliography}{1}

\bibitem{EKM}
R.~Elman, N.~Karpenko, and A.~Merkurjev.
\newblock {\em {The Algebraic and Geometric Theory of Quadratic Forms}}.
\newblock American Mathematical Society, 2008.

\bibitem{hhfba}
N.~Karpenko.
\newblock {\em {Hyperbolicity of hermitian forms over biquaternion algebras}}.
\newblock 2009.

\end{thebibliography}
\addcontentsline{toc}{section}{Références}

\end{document}